\documentclass[12pt,reqno]{amsart}

\usepackage{times}
\usepackage[T1]{fontenc}
\usepackage{mathrsfs}
\usepackage{latexsym}
\usepackage[dvips]{graphics}
\usepackage[titletoc,title]{appendix}
\usepackage{epsfig}

\usepackage{amsmath,amsfonts,amsthm,amssymb,amscd}
\input amssym.def
\input amssym.tex
\usepackage{color}
\usepackage{hyperref}

\usepackage{color}
\usepackage{breakurl}
\usepackage{comment}
\newcommand{\bburl}[1]{\textcolor{blue}{\url{#1}}}

\newcommand{\be}{\begin{equation}}
\newcommand{\ee}{\end{equation}}
\newcommand{\bea}{\begin{eqnarray}}
\newcommand{\eea}{\end{eqnarray}}

\newtheorem{thm}{Theorem}[section]

\newtheorem{lem}[thm]{Lemma}
\newtheorem{prop}[thm]{Proposition}

\newtheorem{rek}[thm]{Remark}

\numberwithin{equation}{section}

\begin{document}

\title{The Schreier space does not have the uniform $\lambda$-property}

\author{Kevin Beanland}
\author{H\`ung Vi\d{\^e}t Chu}
\email{\textcolor{blue}{\href{mailto:beanlandk@wlu.edu}{beanlandk@wlu.edu}}}
\email{\textcolor{blue}{\href{mailto:chuh19@mail.wlu.edu}{chuh19@mail.wlu.edu}}}
\address{Department of Mathematics, Washington and Lee University, Lexington, VA 24450, USA}

\thanks{H.V. Chu is an undergraduate student at Washington \& Lee University.}

\begin{abstract}
The $\lambda$-property and the uniform $\lambda$-property were first introduced by R. Aron and R. Lohman in 1987 as geometric properties of Banach spaces. In 1989, Th. Shura and D. Trautman showed that the Schreier space posseses the $\lambda$-property and asked if it has the uniform $\lambda$-property. In this paper, we show that Schreier space does not have the uniform $\lambda$-property. Furthermore, we show that the dual of the Schreier space does not have the uniform $\lambda$-property.  
\end{abstract}

\maketitle

\section{Introduction}
\subsection{Background}
Schreier space is a Banach space constructed by J. Schreier in 1930 \cite{Schreier1930} as a counterexample to a question of Banach and Saks. The space has the property that the standard unit vector basis (which we denote $(e_i)_{i=1}^\infty$) is weakly null, but there is no subsequence that Cesaro sums to $0$. Recall that the Schreier space is defined as follows: Let $c_{00}$ be the vector space of real sequences $x = (x(1),x(2),x(3),\ldots)$, which are finitely supported. A set $F\subset \mathbb{N}$ is admissible if $\min F\geqslant |F|$. We denote $\mathcal{S}_1$ to be the collection of all admissible subsets of $\mathbb{N}$. The Schreier space is the completion of $c_{00}$ with respect to the following norm
$$||x|| \ =\ \sup_{F\in \mathcal{S}_1} \sum_{i\in F}|x(i)|.$$
\noindent Since the only Banach spaces we refer to this paper are the Schreier space and its dual, we will denote the Schreier space $X$ and its dual $X^*$. Let $Ba(X)$, $S(X)$, and $E(X)$ denote the unit ball, unit sphere, and the set of extreme points of the ball of $X$, respectively. Recall that
$x\in E(X)$ if $x\in S(X)$ and whenever $x = 1/2(y + z)$ for some $y, z\in S(X)$, we have $x = y = z$.

In 1987, R. Aron and R. Lohman introduced geometric properties for Banach spaces, called the $\lambda$-property and the uniform $\lambda$-property \cite{AronLohman}. Since then, much progress have been made in understanding these properties for different spaces \cite{ABC-Schreier,ALS, Bogachev,Bohonos,LohmanAspectsLambda,LohmanDirectSum, PeraltaAntonio, ShuraTrautman}. In 1989, Th. Shura and D. Trautman proved that the Schreier space has the $\lambda$-property and asked if it has the uniform $\lambda$-property \cite{ShuraTrautman}. Recently L. Antunes and the authors of the present paper extended these results for higher order Schreier spaces and their $p$-convexifications \cite{ABC-Schreier}. The first main result of this paper solves the problem first stated in \cite{ShuraTrautman} and restated in \cite{ABC-Schreier}.

\begin{thm}\label{maintheo}
The Schreier space $X$ does not have the uniform $\lambda$-property.
\end{thm}

\noindent The main difficulty in proving that $X$ does not have the uniform $\lambda$-property lies in the fact that we do not have a  workable characterization of $E(X)$. Th. Shura and D. Trautman proved that every element of $E(X)$ is finitely supported; however, even in the case of moderately sized supports (say the max of the support is $10$), it is computationally difficult to classify all of the extreme points. On the other hand, in many of the previous results proving or disproving that a space has the uniform $\lambda$-property, such a characterization is used in an essential way. The main illustrative example is in showing that the space $\ell_1$ does not have the uniform $\lambda$-property \cite{AronLohman}. Since we do not have a full characterization of $E(X)$, we have instead to develop a different technique. 


Our second main result is showing that $X^*$ does not have the uniform $\lambda$-property neither. In \cite{ABC-Schreier}, $X^*$ is shown to have the $\lambda$-property and the set $E(X^*)$ is characterized. 

\begin{thm}\label{dualtheo}
The dual of the Schreier space $X^*$ does not have the uniform $\lambda$-property.
\end{thm}

\subsection{Definition and notation}

A space $Y$ is said to have the $\lambda$-property if for all $y$ in $Ba(Y)$, there exists $0<\lambda\leqslant 1$ such that $y$ can be written as $\lambda e+(1-\lambda)z$ for some $e \in E(Y)$, $z \in Ba(Y)$. We write $(e,z,\lambda)\sim y$ to mean that the vector $y$ can be written in terms of $\lambda$, $e$, and $z$.  Given a vector $y$, we may find different sets $(e,z,\lambda)$ such that $(e,z,\lambda)\sim y$. Because $\lambda$ is bounded above by $1$, we can define the so-called $\lambda$-function: given $y\in Ba(Y)$,
$$\lambda(y)=\sup\{\lambda:(e,z,\lambda)\sim y\}.$$ 
If there exists $\lambda_0>0$ such that for all $y\in Ba(Y)$, $\lambda(y)\geqslant \lambda_0$, we say that $Y$ has the uniform $\lambda$-property. 

We go back to our Schreier space $X$. For each $x \in S(X)$, let 
$$\mathcal{F}_x \ =\ \big\{F\in \mathcal{S}_1: \sum_{i\in F}|x(i)| = 1 \mbox{ and } |x(i)| > 0 \mbox{ for all } i\in F\big\}.$$
We call sets in $\mathcal{F}_x$ the $1$-sets of $x$. Let 
$$\mathcal{A}_x \ =\ \big\{F\in \mathcal{S}_1: \sum_{i\in F}|x(i)| = 1 \big\}.$$
Clearly, $\mathcal{F}_x \subset \mathcal{A}_x$ for each $x \in S(X)$. The standard unit vector basis of $c_{00}$ is denoted $(e_i)_{i=1}^\infty$ and it is a $1$-unconditional Schauder basis for $X$.

The paper is structured as follows: Section \ref{prelim} mentions some essential results for our later proofs; Section \ref{theo1} and Section \ref{theo2} present the proofs of Theorem \ref{maintheo} and Theorem \ref{dualtheo}, respectively; Section \ref{prob} gives some open problems for future research. The Appendix discusses the key technique in our proof of Theorem \ref{maintheo} that helps us solve the problem without a full characterization of $E(X)$. As Th. Shura and D. Trautman \cite{ShuraTrautman} claimed without proof that the cardinality of the support of each vector $x\in E(X)$ is even, we provide a stronger result; that is, the cardinality of the so-called non-maximal 1-set is one-half the cardinality of $x$. We present the proof of this necessary condition in the Appendix as well.

\section{Preliminary results} \label{prelim}
We mention several important results that are useful in our proof of Theorem \ref{maintheo}.
\begin{rek}
Let $x \in S(X)$ and suppose $x\sim (e,y,\lambda)$ with $\lambda < 1$. If $F \in \mathcal{F}_x$, then $F\in \mathcal{A}_e\cap \mathcal{A}_y$. Indeed, if $F \in \mathcal{F}_x$, we have
$$1\ =\ \sum_{i \in F}|x(i)|\ =\ \sum_{i \in F} |\lambda e(i) +(1-\lambda) y(i)|\ \leqslant\ \lambda \sum_{i \in F} | e(i)| + (1-\lambda) \sum_{i \in F} |y(i)| \ \leqslant\ 1.$$ \label{one sets}
Therefore, $\sum_{i\in F} |e(i)|= \sum_{i\in F}|y(i)| = 1$. 
\end{rek}
The following lemma is proved in \cite[Lemma 2.5]{BeanlandCombinatorial}. 
\begin{lem} \label{epsilon-gap}Given $x\in S(X)$, there exists $\varepsilon_x > 0$ such that for all $F\in \mathcal{S}_1\backslash \mathcal{A}_x$, \begin{equation*}\sum_{i\in F}|x(i)|\ <\ 1-\varepsilon_x.\end{equation*}
\end{lem}
We call $\varepsilon_x$ the $\varepsilon$-gap of $x$ and is useful in proving the next proposition. Given a vector $x$, Lemma \ref{epsilon-gap} implies that there is no sequence $(F_n)_{n=1}^\infty\subseteq \mathcal{S}_1$ such that $\sum_{i\in F_n}|x(i)|<1$ for each $F_n$ while $\lim_{n\rightarrow \infty}\sum_{i\in F_n}|x(i)| = 1$. 
\begin{prop}\label{coeffin1set}
If $e \in E(X)$, then for each $i\in \mathbb{N}$, there exists a set $F \in \mathcal{A}_e$ with $i \in F$. \label{in support}
\end{prop}
\begin{proof}
Suppose that there exists some $N\in\mathbb{N}$ such that for all $F\in\mathcal{A}_e$, $N \notin F$. Hence, for all $F\in \mathcal{S}_1$ containing $N$, $F\notin \mathcal{A}_e$ and so, $\sum_{i\in F}|e(i)| < 1-\varepsilon_e$, where $\varepsilon_e$ is the $\varepsilon$-gap of $e$. Form $x$ such that $x(i) = e(i)$ for all $i\neq N$, and $x(N) = e(N) + \varepsilon_e$. Form $y$ such that $y(i) = e(i)$ for all $i\neq N$, and $y(N) = e(N) - \varepsilon_e$. Clearly, $e = \frac{1}{2}(x + y)$. Let $F\in\mathcal{S}_1$. If $N\notin F$, $\sum_{i\in F}|x(i)| \leqslant ||x|| = 1$. If $N\in F$, we have 
\begin{align*}\sum_{i\in F}|x(i)| \ =\ \sum_{i\in F, i\neq N}|x(i)| + |x(N)|
&\ \leqslant \ \sum_{i\in F, i\neq N}|e(i)| + |e(N)| + \varepsilon_e\\ &\ <\ 1-\varepsilon_e + \varepsilon_e \ =\ 1.
\end{align*}
Therefore, $x\in Ba(X)$ and similarly, $y\in Ba(X)$. Because $x\neq y$, we have $e\not\in E(X)$, a contradiction. So, for all $i\in \mathbb{N}$, there exists $F\in \mathcal{A}_e$ with $i\in F$.
\end{proof}

As the basis $(e_i)_{i=1}^\infty$ is a $1$-unconditional basis for $X$, the sign-changing operator is an isometry for $X$, and so if $e \in E(X)$ then $|e|= \sum_{i=1}^\infty |e(i)|e_i\in E(X)$. To show that $X$ does not have the uniform $\lambda$-property, we will find suitable vectors $x$ so that if $x \sim (e,y,\lambda)$ then $\lambda$ satisfies a particular upper bound. The next lemma states that if the coefficients of $x$ are non-negative and $x \sim (e,y,\lambda)$ then $x \sim (|e|,y',\lambda)$ for some $y'\in Ba(X)$. Therefore, when finding an upper bound on $\lambda$, we may assume that the extreme point in the triple has non-negative coefficients. 

\begin{lem}
Let $x\in S(X)$ with $x(i) \geqslant 0$ for each $i\in\mathbb{N}$. If $x\sim (e,y,\lambda)$, then $x\sim (|e|, y', \lambda)$, where $|e| = \sum_{i=1}^\infty |e(i)|e_i$ and for some $y'\in Ba(X)$. \label{positive} 
\end{lem}
\begin{proof}
Note that if $e \in E(X)$, then $|e| \in E(X)$. We have $$|y'(i)| \ =\ \bigg|\frac{x(i) - \lambda |e(i)|}{1-\lambda}\bigg| \ \leqslant \ \bigg|\frac{x(i) - \lambda e(i)}{1-\lambda}\bigg| \ =\ |y(i)|.$$
Hence, for each $i$, $|y'(i)|\leqslant |y(i)|$. Because $y\in Ba(X)$, we know that $y'\in Ba(X)$. 
\end{proof}

\section{$X$ does not have the uniform $\lambda$-property} \label{theo1}

\begin{proof}[Proof of Theorem \ref{maintheo}]
Fix a sufficiently large $n\in \mathbb{N}$. In particular, our $n$ must be large enough so that $1-\frac{2}{n}+\frac{2}{n^3}>\frac{1}{n}-\frac{1}{n^3}>\frac{1}{n^2}$ and also its magnitude must satisfy other conditions made in our arguments below. Consider the following vector
$$x \ =\ \bigg(1,\frac{2}{n}-\frac{2}{n^3},1-\frac{2}{n}+\frac{2}{n^3},\underbrace{0,\ldots,0}_{n-2},\underbrace{\frac{1}{n}-\frac{1}{n^3},\ldots,\frac{1}{n}-\frac{1}{n^3}}_{n},\frac{1}{n^2},0,\ldots\bigg).$$
Note that the last nonzero coefficient is the smallest coeffcient. For notational simplicity, denote $$A \ =\  \{2,3\}, B \ =\ \{4,\ldots,n+1\}, C \ =\ \{n+2,\ldots,2n+2\},$$ 
$$E \ =\ B\backslash \{n+1\}, D \ =\  C\backslash \{2n+2\}.$$ 
A straightforward computation checks that $\|x\| =1$. In particular, all the $1$-sets of $x$ are $A$, $C$, and $\{3,i,j\}$ for any $i\neq j$ in $D$. Because $4\notin F$ for all $F\in \mathcal{A}_x$, by Proposition \ref{coeffin1set}, $x\notin E(X)$. Suppose that $x\sim (e, y, \lambda)$ with $\lambda<1$ (due to $x\notin E(X)$). Using Lemma \ref{positive}, we may assume that $e$ has non-negative coefficients. Trivially, for each $i\in \mathbb{N}$,
\begin{equation}
    y(i) \ =\ \frac{x(i) - \lambda e(i)}{1-\lambda}. \label{the y}
\end{equation}
The proof proceeds by eliminating several possible values for the coefficients of $e$ and then proving that the remaining values yield  $\lambda \leqslant f(n)$ for some function $f$ satisfying $\lim_{n\to \infty} f(n)=0$. This clearly implies that $X$ has does not have the uniform $\lambda$-property. We will repeatedly use Remark \ref{one sets} which states that $\mathcal{F}_x \subset \mathcal{A}_e$. 
The following claims give information about the coefficients of $e$.
\begin{itemize}
    \item[(i)] For $j > 2n+2$, $e(j)=0$.
    \item[(ii)] There is an $\alpha>0 $ so that $e(i)=\alpha$ for $i \in D\cup E$. 
    \item[(iii)] $e(3)>e(2)>\alpha\geqslant e(2n+2)$, where $\alpha$ is defined in item (ii). Also, $e(n+1) = e(2n+2)$. 
\end{itemize}

\noindent Proof of (i): Because $C\in \mathcal{F}_x$, $C \in \mathcal{A}_e$. Since $C \cup \{j\} \in \mathcal{S}_1$ for each $j > 2n+2$ we have that $\sum_{i \in C}|e(i)| + |e(j)| \leqslant \|e\|=1$. Therefore, $e(j)=0$. \newline 

\noindent Proof of (ii): Let $i, j, k$ (pairwise distinct) in $D$. Because $\{3,i,k\},\{3,j,k\}\in\mathcal{F}_x$, $\{3,i,k\},\{3,j,k\} \in \mathcal{A}_e$. It follows that $|e(i)|=|e(j)|$ and since the both are non-negative, $e(i)=e(j)$. Let $\alpha$ be their common value.  So, for all $i\in D$, $e(i)=\alpha$ for some $\alpha\geqslant 0$. Assume that $\alpha = 0$. Because $C\in \mathcal{A}_e$, $e(2n+2) = 1$. Since $A\in \mathcal{A}_e$, we have $e(2)+e(3) = 1$. So, either $e(2) + e(2n+2)>1$ or $e(3)+e(2n+2) > 1$, a contradiction. Therefore, $\alpha > 0$.

Now, we want to show that $e(i) = \alpha$ for each $i \in B\backslash\{n+1\}$. By Proposition \ref{in support}, for $i \in E$, there exists an $F \in \mathcal{A}_e$ with $i \in F$. If we have $e(k)<\alpha$ for some $k\in E$ then $F' = (F \setminus \{k\}) \cup \{j\} \in \mathcal{S}_1$ for any $j \in D\backslash F$. Note that $D\backslash F$ is non-empty because $F$ containing $k$ can have at most $n-1$ coefficients in $D$. This contradicts $||e|| = 1$ because $\sum_{i\in F'}|e(i)|>1$.  If $e(k)> \alpha$ for some $k \in B$, we have the contradiction: $e(3)+e(k)+e(i) > e(3)+e(j)+e(i)=1$ for any $i,j \in D$. Therefore, $e(i) = \alpha$ for each $i \in B\backslash\{n+1\}$, as desired. \newline

\noindent Proof of (iii): The proof is straightforward. Because $e(2) + e(3) = 1$ and $e(3) + 2\alpha = 1$, $e(2) = 2\alpha > \alpha$. Next, suppose that $e(2n+2)> \alpha$. Then $e(3) + e(n+2) + e(2n+2) > e(3) + 2\alpha  = 1$, a contradiction. Hence, $e(2n+2)\leqslant \alpha$. We have shown that $e(2) > \alpha \geqslant e(2n+2)$. It remains to show that $e(3) > e(2)$. If $e(3)\leqslant e(2) = 2\alpha$, then $1 = e(2) + e(3)\leqslant 2\alpha + 2\alpha = 4\alpha$. So, $\alpha \geqslant 1/4$, which implies that $\sum_{i\in C}|e(i)|>1$ for sufficiently large $n$, which is a contradiction. Therefore, we have $e(3)>e(2)>\alpha\geqslant e(2n+2)$.

Finally, if $e(n+1) > e(2n+2)$, then $$\sum_{i\in \{n+1\}\cup D}|e(i)| = e(n+1) + n\alpha \ >\  e(2n+2) + n\alpha = 1,$$
a contradiction. Suppose that $e(n+1)<e(2n+2)\leqslant \alpha$. By Proposition \ref{coeffin1set}, there exists $F\in\mathcal{A}_e$ with $n+1\in F$. We can replace $n+1$ in $F$ by any number in $C$ to have a norm greater than $1$, a contradiction. Hence, $e(n+1) = e(2n+2)$. 

The next claim provides an upper bound for $\lambda$. 
\begin{itemize}
    \item[(iv)] If $e(2n+2)=\alpha$, then $$\lambda \ \leqslant\  \frac{n+1}{n^2}.$$ 
\end{itemize}

\noindent Proof of (iv): If $e(2n+2) = \alpha$, we have $\alpha = \frac{1}{n+1}$ because $C\in \mathcal{A}_e$. For $i\in D$, Equation \ref{the y} gives
$$y(i) \ =\  \frac{\frac{1}{n}-\frac{1}{n^3}-\lambda\frac{1}{n+1}}{1-\lambda}.$$
Note that because $\lambda<1$, we have $\frac{1}{n}-\frac{1}{n^3} -\lambda\frac{1}{n+1}>\frac{1}{n}-\frac{1}{n^3}-\frac{1}{n+1} = \frac{n^2-n-1}{n^3(n+1)}>0$ for $n$ sufficiently large. Hence, for each $i\in D$, $y(i)>0$. Because $y\in Ba(X)$, 

$$\sum_{i\in D} |y(i)| \ =\ \sum_{i\in D}  \frac{\frac{1}{n}-\frac{1}{n^3}-\lambda\frac{1}{n+1}}{1-\lambda}\ =\ n\bigg( \frac{\frac{1}{n}-\frac{1}{n^3}-\lambda\frac{1}{n+1}}{1-\lambda}\bigg)\ \leqslant \ 1.$$
Solving for $\lambda$, we have the upper bound for $\lambda$ $$\lambda \ \leqslant\  \frac{n+1}{n^2},$$ as desired. 

Due to (iii) and (iv), we narrow down to the case $e(3)>e(2)>\alpha> e(2n+2) = e(n+1)$. We will show that in this case $e \not\in E(X)$. Pick $\varepsilon> 0$ such that 
\begin{align}e(3) - 4\varepsilon >e(2) > \alpha  > e(2n+2) + (n+1)\varepsilon \mbox{ and } e(2n+2)-(n+1)\varepsilon > 0,\label{choiceepsilon}\end{align}
and let 
$$z \ =\ (0,2\varepsilon, -2\varepsilon, \underbrace{\varepsilon, \ldots, \varepsilon}_{n-3},-n\varepsilon, \underbrace{\varepsilon,\ldots,\varepsilon}_{n}, -n\varepsilon, 0,\ldots). $$
It is clear that $e = \frac{1}{2}((e+z)+(e-z))$ and $e+z\neq e-z$. If we can show that $e+z$ and $e-z$ are in $Ba(X)$, then $e$ is not an extreme point. 
First, we consider $e+z$. Let $F\in \mathcal{S}_1$ be chosen. We have $e+z$ is 
\begin{align*}(1,e(2)+2\varepsilon,e(3)-2\varepsilon,\underbrace{(\alpha+\varepsilon)}_{n-3 \mbox{ times}},e(n+1)-n\varepsilon, \underbrace{(\alpha+\varepsilon)}_{n\mbox{ times}},e(2n+2)-n\varepsilon, 0,\ldots).\end{align*}
Due to \ref{choiceepsilon}, all coefficients of $e+z$ are non-negative. We proceed by case analysis:
\begin{itemize}
    \item If $\min F = 2$, condition \ref{choiceepsilon} guarantees that $(e+z)(3)$ is the maximum among the remaining coefficients. Hence, $\sum_{i\in F}|(e+z)(i)|\leqslant (e+z)(2)+(e+z)(3) = (e(2)+2\varepsilon)+(e(3)-2\varepsilon) = e(2) + e(3) = 1$. 
    \item If $\min F = 3$, condition \ref{choiceepsilon} guarantees that the two maximum coefficients after the third coefficient is $\alpha+\varepsilon$. Hence, $\sum_{i\in F}|(e+z)(i)|\leqslant (e+z)(3) + 2(\alpha+\varepsilon) = e(3) + 2\alpha = 1$.
    \item If $\min F> 3$, then we can take at most $n$ coefficients of value $\alpha + \varepsilon$ and either the coefficient $e(n+1) - n\varepsilon$ or $e(2n+2) - n\varepsilon$. Again, $n(\alpha+\varepsilon) + e(2n+2) - n\varepsilon = n\alpha + e(2n+2) = 1$. 
\end{itemize}
Therefore, $||e+z|| \leqslant 1$, as desired. Finally, we prove that $||e-z||$ is also at most $1$. We have
\begin{align*}(1,e(2)-2\varepsilon,e(3)+2\varepsilon,\underbrace{(\alpha-\varepsilon)}_{n-3 \mbox{ times}},e(n+1)+n\varepsilon, \underbrace{(\alpha-\varepsilon)}_{n\mbox{ times}},e(2n+2)+n\varepsilon, 0,\ldots).\end{align*}
Due to \ref{choiceepsilon}, all coefficients of $e+z$ are non-negative. We proceed by case analysis:
\begin{itemize}
    \item If $\min F = 2$, condition \ref{choiceepsilon} guarantees that $(e+z)(3)$ is the maximum among the remaining coefficients. Hence, $\sum_{i\in F}|(e+z)(i)|\leqslant (e+z)(2)+(e+z)(3) = (e(2)-2\varepsilon)+(e(3)+2\varepsilon) = e(2) + e(3) = 1$. 
    \item If $\min F = 3$, condition \ref{choiceepsilon} guarantees that the two maximum coefficients after the third coefficient is $\alpha-\varepsilon$. Hence, $\sum_{i\in F}|(e+z)(i)|\leqslant (e+z)(3) + 2(\alpha-\varepsilon) = e(3) + 2\alpha = 1$.
    \item If $\min F> 3$, then we can take at most $n$ coefficients of value $\alpha - \varepsilon$ and the coefficient $e(2n+2) + n\varepsilon$. Again, $n(\alpha-\varepsilon) + e(2n+2) + n\varepsilon = n\alpha + e(2n+2) = 1$. 
\end{itemize}
Therefore, $||e+z|| \leqslant 1$, as desired. We have proved that $e$ is not an extreme point if $e(3)>e(2)>\alpha> e(2n+2) = e(n+1)$.
Putting all cases together, we have
$$\lambda \ \leqslant\  f(n): = \frac{n+1}{n^2}.$$
Since $f(n) \to 0$ as $n \to \infty$, $X$ does not have the uniform $\lambda$-property.
\end{proof}
\section{$X^*$ does not have the uniform $\lambda$-property} \label{theo2}

\begin{proof}[Proof of Theorem \ref{dualtheo}]
In \cite{ABC-Schreier}, it is shown that 
$$E(X^*)\ =\ \big\{\sum_{i \in F} \varepsilon_i e_i : |F|=\min F,\,\varepsilon_i \in \{\pm 1\}, i \in F\big\}.$$
Fix $n \in \mathbb{N}$. For each $k\in \mathbb{N}$, let $F_k=\{2^{k-1},\ldots,2^k-1\}$, $x^*_k = \sum_{i \in F_k} e^*_i$, $x_k=\frac{1}{2^{k-1}}\sum_{i \in F_k} e_i$, and $x=\sum_{k=1}^n x_k$. Using \cite[Proposition 0.7]{CS-book}, we have $\|\sum_{k=1}^n x_k\|=1$. Furthermore $x^*=\frac{1}{n}\sum_{k=1}^n x^*_k\in S(X^*)$ as $x^*(\sum_{k=1}^n x_k)=1$ and $\|\frac{1}{n}\sum_{k=1}^n x^*_k\|\leqslant \frac{1}{n} \sum_{k=1}^n \|x^*_k\| =1$. Find a triple $(e^*,y^*,\lambda) \sim x^*$. Let $F \subset \mathbb{N}$ with $|F|=\min F$ so that $e^* =\sum_{i \in F} \varepsilon_i e^*_i$. Suppose we can find $m \in \{1,\ldots, n \}$ such that $\min F \in F_m$. Then $|F|<2^m$. Note that 
$$|F_m\setminus F\ |\geqslant\ 0 , |F_{m+1}\setminus F|\ >\ 0, \mbox{ and for $k>m+1$}, |F_{k}\setminus F|\ >\ 2^{k-1}-2^m.$$ 
Let $x'=x|_{\mathbb{N}\setminus F}$ (that is, the vector $x$ restricted to the coordinates in $\mathbb{N}\setminus F$.). We know that $\|x'\|\leqslant 1$. By a direct computation (see (\ref{the y}), $y^*(e_i)= 1/(n(1-\lambda))$ for each $i \in \mathbb{N}\setminus F$. We will estimate $y^*(x')$ from below to get an upper bound on $\lambda$. 
\begin{equation}
    \begin{split}
        y^*(x')  &\ =\ y^*(\sum_{k=1}^{m-1} x_k) + y^*(\sum_{k=m}^n x_k|_{\mathbb{N}\setminus F})\\
        & \ \geqslant\ \frac{m-1}{n(1-\lambda)} + \frac{1}{n(1-\lambda)} \sum_{k=m}^n \frac{|F_k \setminus F|}{2^{k-1}} \\
        & \ >\ \frac{m-1}{n(1-\lambda)} + \frac{1}{n(1-\lambda)} \sum_{k=m+2}^n \bigg(1- \frac{1}{2^{k-m-1}}\bigg) \\
        & \ >\ \frac{m-1}{n(1-\lambda)} + \frac{1}{n(1-\lambda)} \bigg(\sum_{k=m+2}^n 1- \sum_{k=m+2}^\infty\frac{1}{2^{k-m-1}}\bigg)\\
        & \ \geqslant\ \frac{m-1}{n(1-\lambda)} + \frac{n-m-1}{n(1-\lambda)} - \frac{1}{n(1-\lambda)} = \frac{n-3}{n(1-\lambda)}.
    \end{split}
\end{equation}
Since $y^*(x') \leqslant 1$, we have $\lambda < 3/n$. 

If $\min F > 2^n-1$, then no such $m$ exists, $x'=x$, and $1\geqslant y^*(x')\geqslant 1/(1-\lambda)$. Therefore, $\lambda=0$. This is the desired result.
\end{proof}

\section{Future research}\label{prob}
We list some possible questions for future research.
\begin{itemize}
    \item[(I)] Given two sets $E, F$, we write $E<F$ if $\max E<\min F$ and write $n<E$ if $n<\min E$. The Schreier families \cite{AlA-Dissertationes} are defined as follows: letting $\mathcal{S}_0=\{F:|F|\leqslant 1\}$ and supposing that $\mathcal{S}_n$ $(n\in\mathbb{N}\cup\{0\})$ has been defined, we define

\begin{equation*}\mathcal{S}_{n+1}\ =\ \big\{\bigcup^{n}_{i=1} E_i:n\leqslant E_1<E_2<...<E_n \mbox{ are in } \mathcal{S}_n\big\}.\end{equation*} 
For each $\mathcal{S}_n$, we define the Banach space $X_{\mathcal{S}_n}$ as the completion of $c_{00}$ with respect to the following norm \begin{equation*}\|x\|_{\mathcal{S}_n}\ =\ \sup_{F\in\mathcal{S}_n}\sum_{i\in F}|x(i)|.\end{equation*}
Do the higher-order Schreier spaces (and their duals) have the uniform $\lambda$-property? We conjecture negatively. 
\item[(II)] Can $X$ be renormed to have the uniform $\lambda$-property? 
\item [(III)] L. Antunes and the authors of the present paper showed that the $p$-convexification of the higher-order Schreier spaces have the uniform $\lambda$-property with $\lambda_0 = \frac{1}{8}$ \cite{ABC-Schreier}. We made no attempt in sharpening the bound of $\frac{1}{8}$. Can we improve the bound? 
\item[(IV)] Is there a characterization of the elements of $E(X)$?  A necessary condition, due to Shura and Trautman \cite{ShuraTrautman}, states that for each $x\in E(X)$, $x$ must have a non-maximal 1-set and the cardinality of $\mbox{supp } x$ is even. 

\end{itemize}

\section{Appendix}
\subsection{Key technique.}
In the proofs of our two main theorems, we use a specific sequence of vectors $(x_n)_{n=1}^\infty$ with $\lim_{n\rightarrow\infty}\lambda(x_n) = 0$, thus disproving the uniform $\lambda$-property. The proof of Theorem \ref{dualtheo} is simpler than the proof of Theorem \ref{maintheo} because we know the exact structure of $E(X^*)$, but do not know $E(X)$. For the proof of Theorem \ref{maintheo}, we construct a sequence $(x_n)$ that may seem unnatural. Here we discuss a lemma that motivates our construction of such a sequence. 

\begin{lem}\label{keytech}
Let $(x_n)$ with $x_n\sim (e_n,y_n,\lambda_n)$ and coefficients of $e_n$ being non-negative. Suppose that there exists $F_n\in \mathcal{S}_1$ such that
\begin{itemize}
    \item $\sum_{i\in F_n}|x_n(i)| = 1-g(n)$,
    \item $\sum_{i\in F_n}|e_n(i)| = 1 - h(n)$,
    \item For all $n\in\mathbb{N}$, $g(n), h(n) > 0$ and $\lim_{n\rightarrow \infty} \frac{g(n)}{h(n)} = 0$.
\end{itemize}  
If $y_n(i)\geqslant 0$ for each $i$, then $\lim_{n\rightarrow\infty}\lambda_n = 0$. 
\end{lem}
\begin{proof}
By Lemma \ref{positive}, we can assume that $e(i)\geqslant 0$ for each $i\in \mathbb{N}$. By Equation \ref{the y} and $y_n(i)\geqslant 0$ for each $i$, we have
\begin{align*} \sum_{i\in F_n}|y_n(i)| &\ =\ \sum_{i\in F_n} \frac{x_n(i)-\lambda_n e_n(i)}{1-\lambda_n} \ =\ \frac{1}{1-\lambda_n}\big(\sum_{i\in F_n}x_n(i)-\lambda_n\sum_{i\in F_n}e_n(i)\big)\\
&\ =\ \frac{1}{1-\lambda_n} (1-g(n)-\lambda_n(1-h(n))) \ \leqslant \ 1.
\end{align*}
Solving for $\lambda_n$, we have $\lambda_n \leqslant \frac{g(n)}{h(n)}$. Because $\lim_{n\rightarrow \infty}\frac{g(n)}{h(n)} = 0$, $\lim_{n\rightarrow\infty }\lambda_n = 0$, as desired. 
\end{proof}
\begin{rek}
The above lemma is a key observation in disproving the uniform $\lambda$-property without a full knowledge of the set of extreme points. In our proof of Theorem \ref{maintheo}, we use $F_n = \{n+2,\ldots, 2n+1\}$, $g(n) = \frac{1}{n^2}$, and $h(n) = \frac{1}{n+1}$. Hence, $\lambda_n\leqslant \frac{n+1}{n^2}$ (see the proof of claim (iv)).
\end{rek}

\subsection{Understanding $E(X)$.}
Let $F\in \mathcal{S}_1$. If $F\cup\{j\}\in\mathcal{S}_1$ for some $j\notin F$, then $F$ is said to be non-maximal. We denote by $\mathcal{S}_1^{MAX}$ the maximal elements of $\mathcal{S}_1$. If $F \in \mathcal{F}_x \setminus \mathcal{S}_1^{MAX}$, then $F$ is a non-maximal 1-set of $x$. 
It is shown in \cite{ShuraTrautman} that for each $x\in E(X)$, $x$ is finitely supported and has a non-maximal 1-set. Denote the support of $x$ by $\mbox{supp}\,x$. 

\begin{lem}\label{ext}
Let $x\in E(X)$. The following hold:
\begin{itemize}
\item [(i)] The vector $x$ has a unique non-maximal 1-set $F$, and $\mbox{supp}\,x\cap [\min\,F,\infty) = F$.
\item [(ii)] Let $F$ be the non-maximal 1-set of $x$. Then $\mbox{supp}\,x\cap [1, |F|] = [1, |F|]$. 
\end{itemize}
\end{lem}

\begin{proof}
We prove item (i). Let $x\in E(X)$. Let $F$ be a non-maximal 1-set of $x$. We claim that $F = \mbox{supp}$ $x\cap [\min F,\infty)$. If not, there exists $j\in \mbox{supp } x \cap [\min F, \infty)$ that is not in $F$. Because $F$ is non-maximal, $F' = F\cup \{j\}\in \mathcal{S}_1$ (note $j> \min F$). We have
$$\sum_{i\in F'} |x(i)|\ =\ \sum_{i\in F}|x(i)| + |x(j)| \ =\ 1+|x(j)| \ >\ 1,$$
a contradiction. 

Suppose $F_1$ and $F_2$ are non-maximal 1-sets of $x$. By above, $F_1\neq F_2$ if and only if $\min F_1 \neq \min F_2$. Assume that $\min F_1 < \min F_2$. If so, we reach the following contradiction
$$\sum_{i\in F_1} |x(i)| \ > \ \sum_{i\in F_2} |x(i)| \ =\ 1.$$
Therefore, $F_1 = F_2$. This completes our proof. Therefore, we can refer to `the non-maximal 1-set' of any vector $x$ instead of 'a non-maximal 1-set'. 

Next, we prove item (ii). Suppose that $x(k) = 0$ for some $k\in [1,|F|]$. By definition, $\min F\geqslant |F|$. By Proposition \ref{coeffin1set}, there exists $F_1\in \mathcal{A}_x$ with $k\in F_1$. If there exists $j\in F\backslash F_1$, then since $F_2 = (F_1\backslash \{k\})\cup \{j\}\in \mathcal{S}_1$, we can sum over these coordinates to get a contradiction. If $F\backslash F_1 = \emptyset$, then $F\subset F_1$, which implies that $\min F_1 \geqslant |F_1| > |F|$. However, $\min F_1 \leqslant |F|$ because $k\in F_1$. We have a contradiction. Therefore, $\mbox{supp }x\cap [1, |F|] = [1, |F|]$.

Finally, it suffices to prove that for any $k\in \mbox{supp}\,x$ and $k\geqslant |F|+1$, $k\in F$. This is obvious since $k\geqslant |F|+1$ guarantees that $F\cup \{k\}\in \mathcal{S}_1$. If $k\notin F$, then $\sum_{i\in F}|x(i)| + |x(k)| > 1$, a contradiction. Hence, $k\in F$. This completes our proof.\end{proof}

The above lemma easily yields the following remark which is a slight strengthening of an unproved statement in \cite{ShuraTrautman}.

\begin{rek}
It follows from item (ii) of Lemma \ref{ext} that the cardinality of the non-maximal 1-set is one-half the support of an extreme point. As a corollary, the support of an extreme point in $E(X)$ is even. 
\end{rek}
\bibliographystyle{abbrv}
\bibliography{main}

\end{document}